\documentclass[12pt]{article}
\usepackage{amsmath,amssymb,amsthm,eucal,color,mathptmx}
\usepackage[top=2cm,right=3cm,left=3cm,bottom=3.5cm]{geometry}

\title{A Cubic Surface of Revolution}
\author{Mark B. Villarino\\[6pt]
Escuela de Matem\'atica, Universidad de Costa Rica,\\
11501 San Jos\'e, Costa Rica}

\date{\today}

\newtheorem{thm}{Theorem}

\makeatletter
\def\section{\@startsection{section}{1}{\z@}{-3.5ex plus -1ex minus
			  -.2ex}{2.3ex plus .2ex}{\large\bf}}
\def\subsection{\@startsection{subsection}{2}{\z@}{-3.25ex plus -1ex
			  minus -.2ex}{1.5ex plus .2ex}{\normalsize\bf}}
\renewcommand{\@dotsep}{200} 
\makeatother

\renewcommand{\leq}{\leqslant}  

\newcommand{\half}{{\mathchoice{\thalf}{\thalf}{\shalf}{\shalf}}}
\newcommand{\shalf}{{\scriptstyle\frac{1}{2}}} 
\newcommand{\thalf}{\tfrac{1}{2}} 

\newcommand{\bC}{\mathbb{C}}        
\newcommand{\bR}{\mathbb{R}}        

\newcommand{\CC}{\mathbf{C}}        
\newcommand{\pp}{\mathbf{p}}        
\newcommand{\rr}{\mathbf{r}}        

\newcommand{\sL}{\mathcal{L}}       
\newcommand{\sS}{\mathcal{S}}       

\newcommand{\word}[1]{\quad\mbox{#1}\quad} 

\newcommand{\url}[1]{{\small\texttt{#1}}} 

\begin{document}

\maketitle

\begin{abstract}
We develop a direct and elementary (calculus-free) exposition of the
famous cubic surface of revolution $x^3 + y^3 + z^3 - 3xyz = 1$.
\end{abstract}



\section{Introduction} 

A well-known exercise in classical differential geometry
\cite{Bar,Dean,Sal} is to show that the set $\sS$ of all points
$(x,y,z) \in \bR^3$ which satisfy the cubic equation
\begin{equation}
\label{eq:S} 
\boxed{F(x,y,z) \equiv x^3 + y^3 + z^3 - 3xyz - 1 = 0}
\end{equation}
is a \emph{surface of revolution}. 

The standard proof (\cite{Dean} and \cite[p.~11]{Sal}), which, in
principle, goes back to \textsc{Lagrange}~\cite{La} and
\textsc{Monge}~\cite{Mo}, is to verify that \eqref{eq:S} satisfies the
partial differential equation (here written as a determinant):
$$
\begin{vmatrix}
F_x(x,y,z) & F_y(x,y,z) & F_z(x,y,z) \\
     x - a &    y - b   & z - c \\
         l &      m     & n
\end{vmatrix} = 0
$$
which characterizes any surface of revolution $F(x,y,z) = 0$ whose
axis of revolution has direction numbers $(l,m,n)$ and goes through
the point $(a,b,c)$. This PDE, for its part, expresses the geometric
property that the normal line through any point of $\sS$ must
intersect the axis of revolution (this is rather subtle;
see~\cite{Gl}). All of this, though perfectly correct, seems
complicated and rather sophisticated just to show that one can obtain
$\sS$ by rotating a suitable curve around a certain fixed line.
Moreover, to carry out this proof one needs to know \emph{a priori}
just what this axis is, something not immediately clear from the
statement of the problem. Nor does the solution give much of a clue as
to \emph{which} curve one rotates.

A search of the literature failed to turn up a treatment of the
problem which differs significantly from that sketched above (although
see~\cite{Bar}).

The polynomial \eqref{eq:S} is quite famous and has been the object of
numerous algebraical and number theoretical investigations. See the
delightful and informative paper~\cite{mch}. See also the remarks at
the end of this paper. It therefore is all the more surprising that an
elementary treatment of its geometrical nature as a surface of
revolution is apparently not to be found in any readily available
source.

Therefore, this paper offers two detailed fully elementary and
calculus-free solutions of the problem based on simple coordinate
geometry. We will obtain a parametric representation of a meridian
curve whose rotation produces $\sS$ as well as a parametric
representation of $\sS$ itself, which we have not seen before
(although it can hardly be new).

Moreover, we relate the surface $\sS$ to the general theory of cubic
surfaces, of which there is an enormous literature (see \cite{lines}),
and in particular we prove a version of the famous Salmon--Cayley
theorem which asserts that any nonsingular (complex) cubic surface
contains $27$ straight lines. 

Finally, this parametrization of $\sS$ and the theory of
\emph{pythagorean triples} will permit us to find infinitely many
\emph{rational} points on the surface $\sS$ defined by the equation~\eqref{eq:S} via our
rational parametrization of~$\sS$.


\section{The first elementary solution} 

The celebrated factorization
\begin{equation}
\label{eq:I} 
x^3 + y^3 + z^3 - 3xyz \equiv(x + y + z)(x^2 + y^2 + z^2 - xy - yz - zx)
\end{equation}
is an endless source of Olympiad problems and is the basis of our first
solution. Let
\begin{equation}
\label{eq:tR} 
t := x + y + z,
\end{equation}
where we assume $t > 0$. Indeed, since the second factor 
in~\eqref{eq:I} is
$$
x^2 + y^2 + z^2 - xy - yz - zx 
= \half(x - y)^2 + \half(x - z)^2 + \half(y - z)^2,
$$
and since we are only interested in finite real points, it cannot be
negative, so \emph{there is no solution to~\eqref{eq:S} with}
$t \leq 0$.

Now we write the equation of $\sS$ in the form
\begin{equation}
\label{eq:sphere} 
x^2 + y^2 + z^2 = \frac{2}{3(x + y + z)} + \frac{(x + y + z)^2}{3}
\equiv \frac{2}{3t} + \frac{t^2}{3}
\end{equation}
which is legitimate because we just showed that $x + y + z \neq 0$
on~$\sS$. This is the equation of a \emph{sphere whose center is at
the origin and whose radius is}
$\sqrt{\frac{2}{3t} + \frac{t^2}{3}}\,$.

Let $\Sigma_t$ denote this sphere~\eqref{eq:sphere} and let $\Pi_t$
denote the plane $x + y + z = t$.  Finally, let
$$
\Gamma_t := \Sigma_t \cap \Pi_t.
$$

We can see that $\Gamma_t$ is nonempty for $t>0$ since the
square of the distance from the origin to $\Pi_t$ is $t^2/3$ and this
is less than the square of the radius $2/3t + t^2/3$ of the
sphere~$\Sigma_t$.

Thus $\Gamma_t$ \emph{is a circle with center on
the line $x = y = z$ and in the plane $\Pi_t$ orthogonal to that line}.
If the intersection were a single point, the distance from the 
origin to the plane would equal the radius of the sphere, that is,
 $t^2/3 = 2/3t + t^2/3$, which is impossible.
Then it follows that $\Gamma_t \subset \sS$ for all $t > 0$ and that, therefore
$$
\sS = \bigcup_{t>0} \Gamma_t \,.
$$

Moreover, the pythagorean theorem now shows that the square of the
radius of $\Gamma_t$ is~$2/3t$. 

Therefore, we have proved the following result.

\begin{thm} 
The surface $\sS$ is a \textbf{surface of revolution} formed by the
union of all the circles with variable center at
$\bigl(\frac{t}{3},\frac{t}{3},\frac{t}{3}\bigr)$, $0 < t < \infty$,
and radius $\sqrt{2/3t}$. Each such circle lies in the plane
$x + y + z = t$, which cuts the corresponding line $x = y = z$
perpendicularly.
\qed
\end{thm}

We add the remark that the equation
$$
x^3 + y^3 + z^3 - r\cdot xyz = 1,
$$
where $r \in \bR$, is a surface of revolution \emph{only for $r = 3$}.
(We prove this later on.) Thus, our equation~\eqref{eq:S} is quite
special.


\section{Parametrizations} 

Now we can parametrize a meridian curve of $\sS$. We recall that the
intersection with $\sS$ of any plane that contains the axis of
revolution of $\sS$ is a meridian curve of $\sS$. The plane
$2z = x + y$ contains the line $x = y = z$ and its normal has
direction numbers $(1,1,-2)$. A unit vector parallel to this normal is
$\bigl(\frac{1}{\sqrt{6}},\frac{1}{\sqrt{6}},-\frac{2}{\sqrt{6}}\bigr)$.
Therefore the vector
$$
\rr_1 := \sqrt{\frac{2}{3t}}\,
\biggl(\frac{1}{\sqrt{6}}, \frac{1}{\sqrt{6}}, -\frac{2}{\sqrt{6}}\biggr)
$$
is a normal vector to $x = y = z$ and its length is the radius of the
circle of intersection. Forming its vector sum with
$\bigl(\frac{t}{3},\frac{t}{3},\frac{t}{3}\bigr)$, we have proved the
following result.

\begin{thm} 
The following parameterization gives us a meridian curve $\CC(t)$
of~$\sS$:
$$
\CC(t) = \biggl( \frac{t}{3} + \frac{1}{3\sqrt{t}},
\frac{t}{3} + \frac{1}{3\sqrt{t}},
\frac{t}{3} - \frac{2}{3\sqrt{t}} \biggr)
$$
where $0 < t < \infty$.
\qed
\end{thm}

A vector perpendicular to $x = y = z$ and to $(1,1,-2)$ simultaneously
is $(1,-1,0)$. A unit vector parallel to it is
$\bigl(\frac{1}{\sqrt{2}},-\frac{1}{\sqrt{2}},0\bigr)$ and
$$
\rr_2 := \sqrt{\dfrac{2}{3t}}\,
\biggl(\frac{1}{\sqrt{2}}, -\frac{1}{\sqrt{2}}, 0\biggr)
$$
is a normal vector to $x = y = z$ whose length is the radius of the
circle of intersection.

Moreover, as we noted earlier, $\rr_1 \perp \rr_2\,$.

Therefore, we can parameterize the surface $\sS$ as follows:
$$
\rr(t,\theta) := \biggl(\frac{t}{3}, \frac{t}{3}, \frac{t}{3}\biggr)
+ \rr_1 \cos\theta + \rr_2 \sin\theta
$$
or, writing
$\rr(t,\theta) := \bigl( x(t,\theta), y(t,\theta), z(t,\theta) \bigr)$,
we have proved the following result.

\begin{thm} 
A coordinate parametrization of~$\sS$ is
\begin{align*}
x(t,\theta) &= \frac{t}{3} + \frac{1}{3\sqrt{t}} \cos\theta
+ \frac{1}{\sqrt{3t}} \sin\theta
\\[\jot]
y(t,\theta) &= \frac{t}{3} + \frac{1}{3\sqrt{t}} \cos\theta
- \frac{1}{\sqrt{3t}} \sin\theta
\\[\jot]
z(t,\theta) &= \frac{t}{3} - \frac{2}{3\sqrt{t}} \cos\theta
\end{align*}
where $0 < t < \infty$ and $0 \leq \theta < 2\pi$.
\qed
\end{thm}

It is a pleasant surprise that the cubic surface $\sS$ has an
elementary parametrization. However, we could have predicted the
existence of such a parameterization \emph{a priori.} For, it is shown
in the general theory of cubic surfaces \cite{Polo} that a real cubic
surface has a rational parametrization over the real numbers if and
only if its real support is a connected set. However, it is not easy
of find such parameterizations and much research has been dedicated to
creating algorithms for producing them (again, see~\cite{Polo}).
We conjecture that the real surface represented by the equation
$$
x^3 + y^3 + z^3 - r\cdot xyz = 1
$$
where $r \in \bR$, is connected if and only if $-\infty < r \leq 3$.
This conjecture seems difficult to prove for $r < 3$ although it is
evident geometrically if one uses computer generated graphs for
suitable values of~$r$. Moreover, the theorem on the existence of
parameterizations for connected cubic surfaces requires 
more advanced techniques than are appropriate for our paper. We add
that it is well known that the trigonometric functions in our
parametrization can be replaced by suitable rational functions of a
single parameter. Indeed, we do so in Section~\ref{sec:rat-solns}
below.


\section{An alternate solution} 

Another way to show that $\sS$ is a surface of revolution is to rotate
the surface $\sS$ around the origin in such a way that the plane
$x + y + z = 0$ becomes the new $XY$-plane, the line $x = y = z$
becomes the new $Z$-axis and the line $x + y = 0$, $z = 0$ becomes the
new $X$-axis. This is accomplished by the rotation equations:
\begin{align*}
x &= \frac{X}{\sqrt{2}} + \frac{Y}{\sqrt{6}} + \frac{Z}{\sqrt{3}} \,,
\\
y &= -\frac{X}{\sqrt{2}} + \frac{Y}{\sqrt{6}} + \frac{Z}{\sqrt{3}} \,,
\\
z &= -\frac{2Y}{\sqrt{6}} + \frac{Z}{\sqrt{3}} \,.
\end{align*}
We find that the surface $\sS$ has the equation
$$
\boxed{Z = \frac{2}{3\sqrt{3}(X^2 + Y^2)}}
$$
which explicitly shows that it is a surface of revolution around the
$Z$-axis obtained by rotating the curve
$Z = \frac{2}{3\sqrt{3}\,Y^2}$.

We remark that the same rotation equations, applied to the equation
$$
x^3 + y^3 + z^3 - r\cdot xyz = 1
$$
where $r \in \bR$, gives an equation which is of the form
$Z = f(X^2 + Y^2)$ if and only if $r = 3$, and therefore
\emph{is not a surface of revolution} with the axis $x=y=z$ for $r \neq 3$.


\section{Singular points of cubic surfaces} 

Let $F(w,x,y,z)$ be an irreducible complex homogeneous polynomial of
(total) degree three in the polynomial ring $\bC[w,x,y,z]$. Then the
point $(w,x,y,z) = (w_0,x_0,y_0,z_0) \equiv \pp$ is a \emph{singular
point} of the algebraic surface $\mathbf{S}$ defined by the equation
$F(w,x,y,z) = 0$ if and only if
$$
\nabla F(\pp) = \mathbf{0} \,;
$$
that is, if and only if
$$
F(\pp) = F_w(\pp) = F_x(\pp) = F_y(\pp) = F_z(\pp) = 0.
$$
The surface $\mathbf{S}$ is called \emph{singular} if it has a
singular point, otherwise it is called \emph{nonsingular.}

Our cubic surface, $\sS$, turns out to be singular. Indeed, the
singular points defined by 
\begin{equation}
\label{eq:Hcubic} 
x^3 + y^3 + z^3 - 3\,xyz = w^3
\end{equation}
are
$$
(0,1,1,1),  \quad
(0,1,\epsilon,\epsilon^2),  \quad
(0,1,\epsilon^2,\epsilon),
$$
where $\epsilon$ is a complex cube root of unity.

The singular points of the rotated form
\begin{equation}
\label{eq:rotated} 
z(x^2 + y^2) - \frac{2}{3\sqrt{3}}\, w^3 = 0
\end{equation}
are
$$
(0,i,1,0),  \quad  (0,-i,1,0),  \quad  (0,0,0,1).
$$

Singular points on a cubic surface can be grouped into different
classes \cite[p.~135]{CN}. 

Suppose that $(w,x,y,z) = (w_0,x_0,y_0,z_0)$ is an isolated singular
point and that the Taylor expansion is:
\begin{align*}
F(w_0, x_0 + x, y_0 + y, z_0 + z)
&= \alpha_{11} x^2 + \alpha_{22} y^2 + \alpha_{33} z^2
\\
&\qquad + 2\alpha_{12} xy + 2\alpha_{13} xz + 2\alpha_{23} yz
+ F_3(x,y,z),
\end{align*}
where $F_3(x,y,z)$ is homogeneous of degree~$3$. The matrix of the
coefficients of the above quadratic form is
$$
\alpha:= \begin{pmatrix}
\alpha_{11} & \alpha_{12} & \alpha_{13} \\
\alpha_{21} & \alpha_{22} & \alpha_{23} \\
\alpha_{31} & \alpha_{32} & \alpha_{33} \end{pmatrix}.
$$
If the determinant of $\alpha$ is zero (i.e., the matrix is singular)
and if the matrix has rank~$2$, then the singular point is
\emph{a biplanar double point} or a~\emph{binode}. The term originates
from the surface having two tangent planes at such a point. For
example, the surface defined by the equation $z^ 3- x^2 + y^2 = 0$ has
a binode at the origin with tangent planes $x \pm y = 0$.

In order to apply these definitions to the surface, we make the
affine change of variable 
$$
X := x + iy,  \quad  Y := x - iy,  \quad  Z := z,  \quad
W := -\biggl( \frac{3\sqrt{3}}{2} \biggr)^{\!1/3} w
$$
in the rotated equation. Then the equation of the 
surface~\eqref{eq:rotated} becomes

\begin{equation}
\label{eq:canon} 
XYZ + W^3 = 0,
\end{equation}
and its singularities are
$$
(0,1,0,0),  \quad  (0,0,1,0),  \quad  (0,0,0,1).
$$
Then, at the point $\pp= (w_0,x_0,y_0,z_0)$ our matrix is 
$$
\alpha = \begin{pmatrix}
0 & z_0 & y_0 \\
z_0 & 0 & y_0 \\
y_0 & x_0 & 0 \end{pmatrix}.
$$

This shows that \emph{each of our singular points is a binode in the
plane} $W = 0$.


\section{The twenty seven lines on a cubic surface} 

In 1849, \textsc{Salmon} \cite[p.~183]{Sal} proved the celebrated
theorem that every nonsingular cubic surface contains twenty seven
(real and/or complex) straight lines. Complete treatises
(see~\cite{Hen}) have been written on this theorem and it continues to
be a source of modern research~\cite{lines}. Sixteen years later,
\textsc{Schl\"afli}~\cite{Sch}, and six years after that,
\textsc{Cayley}~\cite{Cay} wrote long papers, not altogether easy to
read, extending Salmon's theorem to \emph{singular} cubic surfaces.
They classified cubic surfaces into $23$ ``species" according to the
kind of singularities they possess, and found the number of lines
associated with each species. The presence of singularities
\emph{decreases} the number of lines, and the smallest number on a
singular surface is~$3$. Then, more than a century later (1978),
\textsc{Bruce} and \textsc{Wall}~\cite{BW} revisited the work of
Schl\"afli and Cayley using modern techniques and obtained
$21$~species of cubic surfaces.

Since our cubic surface $\sS$ has three binodes as singularities,
Schl\"afli \cite[p.~239]{Sch} and Cayley, \cite[pp.~320--321]{Cay}
placed it in class~XXI while Bruce and Wall \cite[p.~253]{BW} placed
it in class~$3A_2$. Both classifications assign \emph{three lines} to
the surface. If we refer to equation~\eqref{eq:canon}, the three lines
are
\begin{align*}
X = 0, \ W = 0;  &&  Y = 0, \ W = 0;  &&  Z = 0, \ W = 0;
\end{align*}
while the lines for the original cubic \eqref{eq:Hcubic} are
\begin{align*}
x + y + z = 0, \ w = 0;
&& x + \epsilon y + \epsilon^2 z = 0, \ w = 0;
&& x + \epsilon^2 y + \epsilon z = 0, \ w = 0; 
\end{align*}
where $\epsilon$ is a complex cubic root of unity. These are all lines
``at infinity'' since $w = 0$ for all three. Moreover, the first line
is real and the other two are complex. According to Schl\"afli and
Cayley, the first line belongs to subspecies~XXI.1 and the other two
``conjugate'' lines (in their terminology) belong to subspecies~XXI.2.

The proof that these are the \emph{only} lines belonging to our
surface is a consequence of the general theories these authors
develop: but that proof is not easily distilled to our special case.

Thus, it is of interest to directly prove this theorem for our own
special cubic surface~$\sS$. We will first prove the following special
case.

\begin{thm} 
The cubic surface of revolution $\sS$ does not contain any finite real
lines.
\end{thm}

\begin{proof} 
The rotated form of the equation for $\sS$ shows that
$Z>0$ for all (finite) points of the surface.  Therefore, any
line wholly contained in $\sS$ cannot intersect the
$XY$-plane, i.e., it must be \emph{parallel} to the plane $Z = 0$.
Any  such line has the parametric representation
$$
\ell = (at + b, ct + d, e),
$$

where $t \in \bR$ while $a,b,c,d$ and $e > 0$ are real constants.
Since $\ell$ is wholly contained in $\sS$ the three
coordinates in the parametric representation identically satisfy
the equation for $\sS$:
$$
e\bigl[ (at + b)^2 + (ct + d)^2 \bigr] - \frac{2}{3\sqrt{3}} = 0.
$$
If we divide by $e$ and rearrange this in powers of~$t$, we find that
the coefficient of~$t^2$ is
$$
a^2 + c^2.
$$
Since this must hold for all real~$t$, the individual coefficients
must all vanish. In particular, the equation
$$
a^2 + c^2 = 0
$$
must hold. This means that
$$
a = c = 0,
$$
which in turn means that the line is the point $(b,d,e)$. This
contradiction shows that any such line in~$\sS$ must be nonreal.
\end{proof}

We thank the referee for his elegant proof of the following more
general result.

\begin{thm} 
The cubic surface of revolution $\sS$ does not contain any finite real
or complex lines.
\end{thm}

\begin{proof}
Take $W = -1$ in equation \eqref{eq:canon} and write the coordinates as
$(x,y,z)$ instead of $(X,Y,Z)$. Therefore, our equation is
\begin{equation}
\label{eq:xyz} 
xyz = 1.
\end{equation}
Any line $\sL$ in $\sS$ other than the three at infinity must be the
join of two distinct finite points of $\sS$, say $(a,b,c) \in \bC^3$
and $(A,B,C) \in \bC^3$. The general point of~$\sL$ is given by
$$
(1 - \lambda)(a,b,c) + \lambda(A,B,C),  \word{where}  \lambda \in \bC.
$$
Therefore, equation~\eqref{eq:xyz} becomes
\begin{equation}
\label{eq:lxyz} 
\bigl[ \lambda a + (1 - \lambda)A \bigr] \cdot
\bigl[ \lambda b + (1 - \lambda)B \bigr] \cdot
\bigl[ \lambda c + (1 - \lambda)C \bigr] = 1,
\end{equation}
and this must hold for all $\lambda \in \bC$. Dividing \eqref{eq:lxyz}
by $\lambda^3$ and letting $\lambda \to \infty$, we obtain
$$
(a - A)(b - B)(c - C) = 0.
$$
Therefore either $a = A$, or $b = B$, or $c = C$.  

Suppose $a = A$. Then equation~\eqref{eq:lxyz} becomes
\begin{equation}
\label{eq:lxyz2} 
a \cdot \bigl[ \lambda b + (1 - \lambda)B \bigr] \cdot
\bigl[ \lambda c + (1 - \lambda)C \bigr] = 1  
\end{equation}
for all $\lambda \in \bC$. This shows,in particular, that $a \neq 0$.
Divide \eqref{eq:lxyz2} by $\lambda^2$ and let $\lambda \to \infty$.
Then \eqref{eq:lxyz2} becomes
$$
a(b - B)(c - C) = 0.
$$
But $a \neq 0$, which means that either $b = B$ or $c = C$,
necessarily. But either alternative implies the other, which means
that
$$
(a,b,c) = (A,B,C),
$$
and the line $\sL$ collapses to a single point. This contradicts the
assumption that $(a,b,c) \in \bC^3$ and $(A,B,C) \in \bC^3$ are
distinct finite points.

Finally, the alternatives $b = B$ or $c = C$ lead to the same false
conclusion, and therefore there are NO finite lines in~$\sS$.
\end{proof}

We note that both proofs, though based on
totally different ideas, lead to the same contradiction, namely
that the line, supposed to exist, collapses to a single point.


\section{Rational points on a cubic surface} 
\label{sec:rat-solns}
The history of the celebrated problem of finding rational points on a cubic surface is detailed in Chapter XXI of Dickson's monumental work on the history of the theory of numbers \cite{Dick}.  Subsequently the great british mathematician \textsc{L.J. Mordell} made fundamental contributions to solving this problem and he summarized them in his classic book \cite{Mor}.  In particular, on page 82 he states:
\begin{quote}
``No method is known for determining whether rational points  exist on a general cubic surface $f(x,y,z)=0$, or for finding all of them if any exist.  Geometric considerations may prove very helpful and sometimes by their help an infinity of solutions may be found..."

\end{quote}

This statement continues to be true, today.  And, as we will see, ``geometric considerations" will lead us to an infinity of rational points on $\sS$.

If we think of the equation \eqref{eq:S} as an equation in the three
unknowns $(x,y,z)$, we can obtain rational solutions by taking
$$
t =: u^2,  \qquad
\sin\theta =: \frac{2r\sqrt{3}}{r^2 + 3}\,, \qquad
\cos\theta =: \frac{r^2 - 3}{r^2 + 3}\,,
$$
in the parametric representation of $\sS$ where $u$ and $r$ run over all rational numbers.  Then we obtain: 

\begin{thm} 
If $u \neq 0$ and $r$ run over all rational numbers then the following
formulas
\begin{align*}
x &= \frac{u^2}{3} + \frac{1}{3u}\,\frac{r^2 - 3}{r^2 + 3}
+ \frac{1}{u}\,\frac{2r}{r^2 + 3} \,,
\\[\jot]
y &= \frac{u^2}{3} + \frac{1}{3u}\,\frac{r^2 - 3}{r^2 + 3}
- \frac{1}{u}\,\frac{2r}{r^2 + 3} \,,
\\[\jot]
z &= \frac{u^2}{3} - \frac{2}{3u}\,\frac{r^2 - 3}{r^2 + 3} \,,
\end{align*}
furnish infinitely many \textbf{rational points} on the cubic surface defined by$$
x^3 + y^3 + z^3 - 3\cdot xyz = 1.
$$
\end{thm}
\hfill$\Box$

For example, if we take $u = 2$ and $r = \frac{1}{3}$, we obtain the
solution $x = \frac{9}{7}\,$, $y = \frac{15}{14}\,$,
$z = \frac{23}{14}\,$.

\medskip

Our rather \emph{ad hoc} formulas for $\cos\theta$ and $\sin\theta$ are based on the
standard formulas for the \emph{Pythagorean triples}  as applied to the give the complete (positive) rational number solution to the equation  $x^2+y^2=1$.  Namely, the solution to $a^2+b^2=c^2$ given by $a=2mn, b=m^2-n^2, c=m^2+n^2$ $m>n$ where $m$ and $n$ run through all integers becomes $x=\frac{2r}{r^2+1}, y=\frac{r^2-1}{r^2+1}$ where $r=\frac{m}{n}, n\neq 0$ runs through all positive rational numbers and we take $x=\sin\theta, y=\cos\theta.$ In order to cancel the term $\sqrt{3}$ in the denominator of the term multiplying $\sin\theta$ in the coordinate parametrization of $\sS$ we replace the numerator $2r$ by $\sqrt{3}\cdot 2r$.  But, in order to maintain the identity $x^2+y^2=1$ the constant $+1$ in there formulas for $x$ and $y$ must be replaced by $+3$ and our rational number parametrization results.

Although our
 rational parametrization of~$\sS$ gives infinitely many rational
solutions to the cubic equation~\eqref{eq:S}, it does not give all rational solutions.   For example, the solution
\begin{align}
 x=\frac{18}{7},&&y=\frac{16}{7},&&z=\frac{15}{7}
 \end{align}
 is not given by our formulas as the reader an check by eliminating $u$ and using the rational root theorem on the sextic polynomial equation that results for $r$ .

We only mention it to show that our
parametric representation of $\sS$, when conjoined with the famous
formulas for Pythagorean triples, gives us a nice bonus in the form of
a  rational parametrization of $\sS$. The complete rational
solution, as well as references to the work of \textsc{Ramanujan} and
others on this equation can be found in~\cite{gs}.


\subsubsection*{Acknowledgment}
Many thanks to the referee whose detailed review led to significant improvements in exposition and content.
Thanks to Joseph C. V\'arilly for helpful comments.
Support from the Vicerrector\'ia de Investigaci\'on of the 
University of Costa Rica is acknowledged.


\end{document}